\documentclass[a4paper,12pt,reqno]{article}
\usepackage{amsmath,amsfonts,amssymb,amsthm,bbm}
\usepackage{amsmath,amsfonts,amssymb,amsthm,bbm,graphics,epsfig,psfrag,epstopdf,dsfont,esint,mathtools,yhmath}
\usepackage{graphics,epsfig} 
\usepackage{paralist,subfigure}
\usepackage{hyperref} 
\usepackage[active]{srcltx} 
\usepackage{color,soul} 
\usepackage[latin1]{inputenc}
\usepackage[OT1]{fontenc}

\newtheorem{theorem}{Theorem}[section]

\newtheorem{lemma}[theorem]{Lemma}

\newtheorem{open}{Open problem}

\theoremstyle{definition}

\newcommand{\baa}{\begin{array}}
\newcommand{\eaa}{\end{array}}
\newcommand{\ba}{\begin{eqnarray}}
\newcommand{\ea}{\end{eqnarray}}
\newcommand{\be}{\begin{equation}}
\newcommand{\ee}{\end{equation}}

\def\N{\mathbb{N}}

\def\R{\mathbb{R}}

\def\epsilon{\varepsilon}

\let\.=\cdot
\let\0=\emptyset

\def\1{\mathbbm{1}}

\newenvironment{formula}[1]{\begin{equation}\label{#1}}{\end{equation}\noindent}

\def\Fi#1{\begin{formula}{#1}}
\def\Ff{\end{formula}\noindent}

\newcommand{\SE}{\setcounter{equation}{0} \section}

\begin{document}

\title{{\bf A Faber-Krahn inequality for the Laplacian with drift under Robin boundary condition} \thanks{This work has been supported by the French Agence Nationale de la Recherche (ANR), in the framework of the ReaCh project ANR-23-CE40-0023-02.}}

\author{Fran{\c{c}}ois Hamel$^{\,\hbox{\small{a}}}$, Emmanuel Russ$^{\,\hbox{\small{a}}}$\\
\\
\footnotesize{$^{\,\hbox{\small{a}}}$ Aix Marseille Univ, CNRS, Centrale Marseille, I2M, Marseille, France}}
\date{}

\maketitle

\begin{abstract}
We prove a Faber-Krahn inequality for the Laplacian with drift under Robin boundary condition, provided that the $\beta$ parameter in the Robin condition is large enough. The proof relies on a compactness argument, on the convergence of Robin eigenvalues to Dirichlet eigenvalues when $\beta$ goes to infinity, and on a strict Faber-Krahn inequality under Dirichlet boundary condition. We also show the existence and uniqueness of drifts $v$ satisfying some $L^\infty$ constraints and minimizing or maximizing the principal eigenvalue of $-\Delta+v\cdot\nabla$ in a fixed domain and with a fixed parameter $\beta>0$ in the Robin condition.
\end{abstract}

\tableofcontents

\SE{Introduction}\label{intro}

Throughout this paper, $d\ge 1$ is an integer. For all $x\in \R^d$, denote by $\left\vert x\right\vert$ the Euclidean norm of $x$ and define
$$
e_r(x):=\frac{x}{\left\vert x\right\vert}\mbox{ for all }x\in\R^d\setminus \left\{0\right\}.
$$
\noindent Let $\Omega\subset \R^d$ be a bounded domain (connected open set) of class $C^2$, with outward unit normal $\nu$ on $\partial\Omega$. If $v\in L^{\infty}(\Omega,\R^d)$ is a bounded measurable vector field, set
$$
\left\Vert v\right\Vert_\infty:=\left\Vert \left\vert v\right\vert\right\Vert_\infty.
$$

We are interested in the principal eigenvalue of the operator $-\Delta+v\cdot\nabla$ in~$\Omega$ under Robin boundary condition on $\partial\Omega$. More precisely, let $\beta>0$. Then, by~\cite[Theorem~A.4]{YD} and Krein-Rutman theory~\cite{AEG}, let $\lambda_1^{\beta}(\Omega,v)$ denote the principal eigenvalue of the problem
\be\label{ev}
\left\{
\begin{array}{ll}
-\Delta \varphi^\beta_{\Omega,v}+v\cdot\nabla \varphi^\beta_{\Omega,v}=\lambda_1^{\beta}(\Omega,v) \varphi^\beta_{\Omega,v} & \mbox{ in }\Omega,\\
\displaystyle\frac{\partial \varphi^\beta_{\Omega,v}}{\partial\nu}+\beta \varphi^\beta_{\Omega,v}=0 & \mbox{ on }\partial\Omega.
\end{array}
\right.
\ee
This principal eigenvalue is simple, the corresponding eigenfunction $\varphi^\beta_{\Omega,v}$ is positive in~$\Omega$ and none of the other eigenvalues corresponds to a positive eigenfunction (see the discussion after \cite[Theorem 1.3]{YD}). By $W^{2,p}$ elliptic regularity (\cite[Theorem~A.29]{YD}), the function $\varphi^\beta_{\Omega,v}$ belongs to $W^{2,p}(\Omega)$ for all $p\in [1,\infty)$ and then to $C^{1,\alpha}(\overline{\Omega})$ for all $\alpha\in (0,1)$. The first line in~\eqref{ev} is therefore understood almost everywhere in~$\Omega$. The Hopf lemma shows that $\varphi^\beta_{\Omega,v}>0$ on $\partial\Omega$, which, in turn, implies that
$$
\min_{\overline{\Omega}} \varphi^\beta_{\Omega,v}>0.
$$ 
We usually normalize $\varphi^\beta_{\Omega,v}$ by
\begin{equation} \label{eq:normal}
\max_{\overline{\Omega}} \varphi^\beta_{\Omega,v}=1.
\end{equation}
Moreover, there holds
$$\lambda_1^\beta(\Omega,v)>0.$$

Fix $\tau\ge 0$ and $m>0$. We are interested in the infimum of $\lambda_1^\beta(\Omega,v)$ when $\Omega$ and $v$ vary under the constraints $\left\vert \Omega\right\vert=m$ (throughout the paper, $|A|$ denotes the $n$-dimensional Lebesgue measure of $A$ for all measurable sets $A\subset \R^d$) and 
\begin{equation} \label{eq:vtau}
\left\Vert v\right\Vert_\infty\le \tau.
\end{equation}
In the sequel, $\Omega^{\ast}$ stands for the Euclidean ball centered at $0$ such that $\left\vert \Omega^\ast\right\vert=\vert\Omega\vert$.

Our main result states that, when $\Omega$ is not a ball, $\lambda_1^\beta(\Omega,v)$ is (strictly) greater than the corresponding quantity in $\Omega^\ast$ when $v=\tau\,e_r$, provided that $\beta$ is large enough:

\begin{theorem} \label{th:main}
Let $\Omega\subset\R^d$ be a bounded $C^2$ domain and $\tau\ge 0$. Assume that $\Omega$ is not a ball. Then there exist $\beta_0>0$ and $\varepsilon>0$ with the following property:
\be\label{ineqlambdav}
\forall\,\beta\ge \beta_0,\ \forall\,v\in L^{\infty}(\Omega,\R^d)\mbox{ such that } \left\Vert v\right\Vert_\infty\le \tau,\ \lambda_1^\beta(\Omega,v)\ge \lambda_1^\beta(\Omega^\ast,\tau e_r)+\varepsilon.
\ee
\end{theorem}

When $\tau=0$, {\it i.e.} when $-\Delta+v\cdot\nabla=-\Delta$ is merely (minus) the Laplacian, it was proved in \cite{Bo,BD,D,DK} that, for all $\beta>0$,
\be\label{ineqlambda0}
\lambda_1^\beta(\Omega,0)\ge \lambda_1^\beta(\Omega^{\ast},0)
\ee
and equality holds if and only if $\Omega=\Omega^{\ast}$ up to translation. When $\tau\neq0$, Theorem~\ref{th:main} provides on the one hand a quantified strict inequality if $\Omega$ is not a ball, but the conclusion is only established above some threshold for $\beta$, contrary to \cite[Theorem~1.1]{BD}, and it actually can not hold for all $\beta>0$, since
$$
\lim_{\beta\rightarrow 0} \lambda_1^\beta(\Omega,v)=\lim_{\beta\rightarrow 0} \lambda_1^\beta(\Omega^\ast,\tau e_r)=0
$$
for each $v\in L^\infty(\Omega,\R^d)$ (as follows from Lemma~\ref{lem:robindir} below). On the other hand, when $\Omega=\Omega^{\ast}$, the uniqueness part in Theorem~\ref{th:fixedomega} below ensures that, for all $v\in L^{\infty}(\Omega^\ast,\R^d)$ with $\left\Vert v\right\Vert\le\tau$, if $v\ne\tau e_r$, then $\lambda_1^{\beta}(\Omega^\ast,v)>\lambda_1^{\beta}(\Omega^\ast,\tau e_r)$ for all $\beta>0$.

The following question nevertheless remains open:

\begin{open}
Let $\Omega\subset\R^d$ be a bounded $C^{2}$ domain, $\tau\ge 0$ and $v\in L^{\infty}(\Omega,\R^d)$ with $\left\Vert v\right\Vert_\infty\le\tau$. Does the inequality
$$
\lambda_1^\beta(\Omega,v)\ge \lambda_1^\beta(\Omega^{\ast},\tau e_r)
$$
hold for all $\beta>0$$\,$?
\end{open}

Recall that, under Dirichlet boundary condition, it was proved in~\cite[Theorem~1.1]{HNRpreprint} and \cite[Remark 6.9]{HNR} that, whenever \eqref{eq:vtau} holds, 
\be\label{ineqlambdaD}
\lambda_1^D(\Omega,v)\ge \lambda_1^D(\Omega^{\ast},\tau e_r),
\ee
where $\lambda_1^D(\Omega,v)$ stands for the principal eigenvalue of $-\Delta+v\cdot\nabla$ under Dirichlet boundary condition. Moreover, equality holds in~\eqref{ineqlambdaD} if and only if, up to translation, $\Omega=\Omega^{\ast}$ and $v=\tau e_r$. The inequalities~\eqref{ineqlambdav}-\eqref{ineqlambdaD} are called Faber-Krahn type inequalities. This terminology originates from the results of Faber~\cite{F} and Krahn~\cite{K1,K2}, who proved that
$$\lambda_1^D(\Omega,0)\ge\lambda_1^D(\Omega^\ast,0),$$
with equality if and only if, up to translation, $\Omega=\Omega^\ast$. The latter inequality means that a radially symmetric membrane which is fixed at its boundary has the lowest fundamental tone among all equimeasurable membranes, answering a conjecture of Rayleigh~\cite{R} set in dimension $d=2$. Since these pioneering papers, much work has been done on various related optimization eigenvalue problems for elliptic operators, for instance on higher eigenvalues or functions of the eigenvalues of $-\Delta$ under Dirichlet boundary condition~\cite{AB1,AB3,AHS,BH,CO,PPW,P1,P2,WK}, under Neumann boun\-dary condition~\cite{P2,S,W}, or for the first eigenvalue of $\Delta^2$ under boundary conditions $\varphi=\frac{\partial\varphi}{\partial\nu}=0$ on $\partial\Omega$~\cite{AB4,N}. We refer to the surveys~\cite{B,H} for many more references on these topics.

The second main result deals with an optimization problem when the domain $\Omega$ is fixed and $v$ varies under the constraint~\eqref{eq:vtau}. Define, for all $\beta>0$ and $\tau\ge 0$ given:
\be\label{underlambda}
\underline{\lambda}^\beta(\Omega,\tau):=\inf\left\{\lambda_1^\beta(\Omega,v):\left\Vert v\right\Vert_\infty\le \tau\right\}
\ee
and
\be\label{overlambda}
\overline{\lambda}^\beta(\Omega,\tau):=\sup\left\{\lambda_1^\beta(\Omega,v):\left\Vert v\right\Vert_\infty\le \tau\right\}.
\ee
We claim that these lower and upper bounds are positive real numbers, are uniquely reached and provide an identity relating the optimizing vector fields and the corresponding eigenfunctions:

\begin{theorem} \label{th:fixedomega} [Optimization in fixed domains]
Let $\Omega\subset \R^d$ be a bounded $C^2$ domain, $\tau\ge 0$ and $\beta>0$.
\begin{enumerate}
\item There exists a unique $\underline{v}\in L^{\infty}(\Omega,\R^d)$ meeting $\left\Vert\underline{v}\right\Vert_\infty\le \tau$ such that $\underline{\lambda}^\beta(\Omega,\tau)=\lambda_1^\beta(\Omega,\underline{v})$. One has $\left\vert \underline{v}(x)\right\vert=\tau$ for almost every $x\in \Omega$. Moreover, if $\underline{\varphi}:=\varphi^\beta_{\Omega,\underline{v}}$ is the corresponding eigenfunction, then
\begin{equation} \label{eq:minvgradphi}
\underline{v}(x)\cdot\nabla\underline{\varphi}(x)=-\tau|\nabla\underline{\varphi}(x)|\mbox{ for almost every }x\in \Omega.
\end{equation}
Lastly, if $\lambda\in\R$ and $\phi\in\bigcap_{1\le p<\infty}W^{2,p}(\Omega)$ satisfy
\be\label{eqphi}
\left\{
\begin{array}{ll}
-\Delta \phi-\tau\,|\nabla\phi|=\lambda\phi\ \hbox{ and }\ \phi\ge0 & \mbox{in }\Omega,\\
\displaystyle\frac{\partial \phi}{\partial\nu}+\beta \phi=0 & \mbox{on }\partial\Omega,
\end{array}
\right.
\ee
and $\max_{\overline{\Omega}}\phi=1$, then $\lambda=\underline{\lambda}^\beta(\Omega,\tau)$ and $\phi=\underline{\varphi}$ in $\overline{\Omega}$. Notice that, from elliptic regularity theory applied to~\eqref{eqphi}, since $\phi$ and $|\nabla\phi|$ belong to $C^{0,\alpha}(\overline{\Omega})$ for all $\alpha\in(0,1)$, the function $\phi$ belongs to $C^{2,\alpha}_{loc}(\Omega)$ for all $\alpha\in(0,1)$, and the first line of~\eqref{eqphi} holds in the classical sense in $\Omega$.
\item Similarly, there exists a unique $\overline{v}\in L^{\infty}(\Omega,\R^d)$ meeting $\left\Vert\overline{v}\right\Vert_\infty\le \tau$ such that $\overline{\lambda}^\beta(\Omega,\tau)=\lambda_1^\beta(\Omega,\overline{v})$. One has $\left\vert \overline{v}(x)\right\vert=\tau$ for almost every $x\in \Omega$. Moreover, if $\overline{\varphi}:=\varphi^\beta_{\Omega,\overline{v}}$ is the corresponding eigenfunction, then
\begin{equation} \label{eq:maxvgradphi}
\overline{v}(x)\cdot\nabla\overline{\varphi}(x)=\tau\left\vert \nabla\overline{\varphi}(x)\right\vert\mbox{ for almost every }x\in \Omega.
\end{equation}
Lastly, if $\lambda\in\R$ and $\phi\in\bigcap_{1\le p<\infty}W^{2,p}(\Omega)$ satisfy
\be\label{eqphi2}
\left\{
\begin{array}{ll}
-\Delta \phi+\tau\,|\nabla\phi|=\lambda\phi \ \hbox{ and }\ \phi\ge0 & \mbox{in }\Omega,\\
\displaystyle\frac{\partial \phi}{\partial\nu}+\beta \phi=0 & \mbox{on }\partial\Omega,
\end{array}
\right.
\ee
and $\max_{\overline{\Omega}}\phi=1$, then $\lambda=\overline{\lambda}^\beta(\Omega,\tau)$ and $\phi=\overline{\varphi}$ in $\overline{\Omega}$. As for~\eqref{eqphi}, the function~$\phi$ then belongs to  $C^{2,\alpha}_{loc}(\Omega)$ for all $\alpha\in(0,1)$, and the first line of~\eqref{eqphi2} holds in the classical sense in $\Omega$.
\item If $\Omega=\Omega^\ast$, then $\underline{v}=\tau\,e_r$, $\overline{v}=-\tau\,e_r$ in $\Omega^\ast$ and the functions $\underline{\varphi}$ and $\overline{\varphi}$ are radially decreasing in $\Omega^{\ast}$.
\end{enumerate}
\end{theorem}

We point out that similar properties had been derived in~\cite{HNRpreprint,HNR} for the extremal quantities $\underline{\lambda}(\Omega,\tau)$ and $\overline{\lambda}(\Omega,\tau)$ defined like~$\underline{\lambda}^\beta(\Omega,\tau)$ and~$\overline{\lambda}^\beta(\Omega,\tau)$ in~\eqref{underlambda}-\eqref{overlambda} with the Dirichlet eigenvalues $\lambda_1(\Omega,v)$ instead of the Robin ones~$\lambda_1^\beta(\Omega,v)$. The asymptotic behavior as $\tau\to+\infty$ of the eigenfunctions associated with~$\underline{\lambda}(\Omega,\tau)$ was analyzed in~\cite{HRR}.\par

The paper is organized as follows. In Section \ref{sec:compar}, we provide comparisons results between Robin, Dirichlet and Neumann eigenvalues in a fixed domain and for a given drift, and prove convergence of the Robin eigenvalues when $\beta\rightarrow+\infty$ (resp. when $\beta\rightarrow 0$) to the corresponding Dirichlet (resp. Neumann) eigenvalues. Section \ref{sec:fixed} is devoted to the proof of Theorem \ref{th:fixedomega}. Finally, we establish Theorem \ref{th:main} is Section \ref{sec:main}.


\SE{Comparisons and convergence results between Robin, Dirichlet and Neumann principal eigenvalues} \label{sec:compar}

This section is concerned with some comparisons and convergence results for Robin and Dirichlet principal eigenvalues in a given domain $\Omega$. The results will be used in the proofs of the main Theorems~\ref{th:main} and~\ref{th:fixedomega}.
 
We first start with an auxiliary comparison lemma between sub- and super-solutions.

\begin{lemma} \label{lem:compar}
Let $\mu\in \R$, $\beta\ge0$, and $v\in L^\infty(\Omega,\R^d)$. Let $\psi,\varphi\in W^{2,p}(\Omega)$ for all $1\le p<\infty$, such that $\psi\ge0$ and $\varphi\ge 0$ in $\Omega$, $\left\Vert \psi\right\Vert_\infty=\left\Vert \varphi\right\Vert_\infty=1$, and
$$
\left\{
\begin{array}{ll}
-\Delta\psi+v\cdot\nabla\psi\ge \mu\psi & \mbox{ a.e. in }\Omega,\vspace{3pt}\\
-\Delta\varphi+v\cdot\nabla\varphi\le \mu\varphi& \mbox{ a.e. in }\Omega.
\end{array}
\right.
$$
Assume also that 
$$
\frac{\partial\psi}{\partial\nu}+\beta\psi\ge0\ge\frac{\partial\varphi}{\partial\nu}+\beta\varphi\mbox{ on }\partial\Omega.
$$
Then $\psi=\varphi$ in $\overline{\Omega}$.
\end{lemma}

\begin{proof}
The argument is reminiscent of the proof of \cite[Lemma 2.1]{HNRpreprint}. Remember first that $\psi$ and $\varphi$ belong to $C^{1,\alpha}(\overline{\Omega})$ for all $\alpha\in(0,1)$. Furthermore, $\psi>0$ in $\Omega$ from the interior strong maximum principle (otherwise, $\psi$ would be identically $0$ in $\Omega$, contradicting $\|\psi\|_\infty=1$). Observe now that $\psi>0$ on $\partial\Omega$. Indeed, if there exists $x_0\in \partial\Omega$ such that $\psi(x_0)=0$, then the Hopf lemma shows that $\frac{\partial\psi}{\partial\nu}(x_0)<0$, which is impossible by the boundary condition satisfied by $\psi$. Thus, being continuous in $\overline{\Omega}$, $\psi$ is bounded below by a positive constant, so that there exists $\gamma>0$ such that $\gamma\psi>\varphi$ in $\Omega$. Define
$$
\gamma^{\ast}:=\inf\left\{\gamma>0:\gamma\psi>\varphi\mbox{ in }\Omega\right\}
$$
and $w:=\gamma^{\ast}\psi-\varphi$. Note that, since $\varphi\ge0$ in $\Omega$ and $\|\varphi\|_\infty=1$, $\gamma^{\ast}>0$. The function $w$ is nonnegative in $\overline{\Omega}$, 
$$
\frac{\partial w}{\partial\nu}+\beta w\ge0\mbox{ on }\partial\Omega
$$
and
$$
-\Delta w+v\cdot\nabla w-\mu w\ge 0\mbox{ a.e. in }\Omega.
$$
If $w>0$ in $\Omega$, then, as before, $w$ is bounded below by a positive constant in $\Omega$, so there exists $\delta>0$ such that $w>\delta\varphi$ in $\Omega$, which entails in turn
$$
\frac{\gamma^{\ast}}{1+\delta}\,\psi>\varphi\mbox{ in }\Omega,
$$
contradicting the definition of $\gamma^{\ast}$, since $\gamma^{\ast}>0$. Therefore, there exists $x_0\in \Omega$ such that $w(x_0)=0$, and since $w\ge 0$ in $\Omega$, the strong maximum principle indicates that $w(x)=0$ everywhere in $\Omega$ and then in $\overline{\Omega}$ by continuity, meaning that $\gamma^{\ast}\psi=\varphi$ in $\overline{\Omega}$. The condition $\left\Vert \varphi\right\Vert_\infty=\left\Vert \psi\right\Vert_\infty=1$ finally yields $\varphi=\psi$ in $\overline{\Omega}$. 
\end{proof}
 
Let now $\Omega\subset \R^d$ be a bounded $C^2$ domain and $v\in L^\infty(\Omega,\R^d)$. Denote by $\lambda_1^D(\Omega,v)$ the principal eigenvalue of $-\Delta+v\cdot\nabla$ in $\Omega$ under Dirichlet boundary condition and by $\varphi^D_{\Omega,v}\in\bigcap_{1\le p<\infty}W^{2,p}(\Omega)$ the corresponding principal eigenfunction (which is positive in $\Omega$) normalized by
$$\left\Vert \varphi^D_{\Omega,v}\right\Vert_\infty=1.$$
We will show that the map $\beta\mapsto \lambda_1^\beta(\Omega,v)$ is increasing in $(0,\infty)$ and converges to $\lambda_1^D(\Omega,v)$ at infinity, and to $0$ (that is, the principal eigenvalue of $-\Delta+v\cdot\nabla$ in $\Omega$ under Neumann boundary condition) as $\beta\to0$:

\begin{lemma} \label{lem:robindir}
Let $\Omega\subset \R^d$ be a bounded $C^2$ domain and $v\in L^\infty(\Omega,\R^d)$. Then the map $\beta\mapsto \lambda_1^\beta(\Omega,v)$ is increasing in $(0,+\infty)$. Furthermore,
\begin{equation} \label{eq:limlambda1}
\lim_{\beta\rightarrow+\infty} \lambda_1^\beta(\Omega,v)=\lambda_1^D(\Omega,v).
\end{equation}
and
$$\lim_{\beta\rightarrow0} \lambda_1^\beta(\Omega,v)=0.$$
\end{lemma}

\begin{proof}
Let $0<\beta_1<\beta_2$ and assume by way of contradiction that $\lambda_1^{\beta_2}(\Omega,v)\le \lambda_1^{\beta_1}(\Omega,v)$. Set $\varphi_1:=\varphi^{\beta_1}_{\Omega,v}$ and $\varphi_2:=\varphi^{\beta_2}_{\Omega,v}$. Both functions $\varphi_1$ and $\varphi_2$ are positive in $\overline{\Omega}$ and they satisfy
$$
\left\{
\begin{array}{ll}
-\Delta\varphi_1+v\cdot\nabla\varphi_1=\lambda_1^{\beta_1}(\Omega,v)\varphi_1 & \mbox{ a.e. in }\Omega,\vspace{3pt}\\
-\Delta\varphi_2+v\cdot\nabla\varphi_2=\lambda_1^{\beta_2}(\Omega,v)\varphi_2\le\lambda_1^{\beta_1}(\Omega,v)\varphi_2 & \mbox{ a.e. in }\Omega,
\end{array}
\right.
$$
together with
\be\label{strict}
\frac{\partial\varphi_1}{\partial\nu}+\beta_1\varphi_1=0=\frac{\partial\varphi_2}{\partial\nu}+\beta_2\varphi_2>\frac{\partial\varphi_2}{\partial\nu}+\beta_1\varphi_2\mbox{ on }\partial\Omega.
\ee
Lemma~\ref{lem:compar} applied with $(\mu,\beta,\psi,\varphi):=(\lambda_1^{\beta_1}(\Omega,v),\beta_1,\varphi_1,\varphi_2)$ then entails $\varphi_1=\varphi_2$ in~$\overline{\Omega}$, contradicting the strict inequality in~\eqref{strict}. Finally,
$$\lambda_1^{\beta_1}(\Omega,v)<\lambda_1^{\beta_2}(\Omega,v),$$
and the map $\beta\mapsto\lambda_1^\beta(\Omega,v)$ is increasing in $(0,+\infty)$.

Let now $\beta>0$ and assume by way of contradiction that $\lambda_1^{\beta}(\Omega,v)\ge \lambda_1^D(\Omega,v)$. Both functions $\varphi^\beta_{\Omega,v}$ and $\varphi^D_{\Omega,v}$ are positive in $\Omega$ and they satisfy
$$
\left\{
\begin{array}{ll}
-\Delta\varphi^\beta_{\Omega,v}+v\cdot\nabla\varphi^\beta_{\Omega,v}=\lambda_1^{\beta}(\Omega,v)\varphi^\beta_{\Omega,v} & \mbox{ a.e. in }\Omega,\vspace{3pt}\\
-\Delta\varphi^D_{\Omega,v}+v\cdot\nabla\varphi^D_{\Omega,v}=\lambda_1^D(\Omega,v)\varphi^D_{\Omega,v}\le\lambda_1^{\beta}(\Omega,v)\varphi^D_{\Omega,v} & \mbox{ a.e. in }\Omega.
\end{array}
\right.
$$
Furthermore, the Hopf lemma implies that $\frac{\partial\varphi^D_{\Omega,v}}{\partial\nu}<0$ on $\partial\Omega$, whence
\be\label{strict2}
\frac{\partial\varphi^\beta_{\Omega,v}}{\partial\nu}+\beta\varphi^\beta_{\Omega,v}=0>\frac{\partial\varphi^D_{\Omega,v}}{\partial\nu}+\beta\varphi^D_{\Omega,v}\mbox{ on }\partial\Omega.
\ee
Lemma~\ref{lem:compar} applied with $(\mu,\beta,\psi,\varphi):=(\lambda_1^{\beta}(\Omega,v),\beta,\varphi^\beta_{\Omega,v},\varphi^D_{\Omega,v})$ then entails $\varphi^\beta_{\Omega,v}=\varphi^D_{\Omega,v}$ in~$\overline{\Omega}$, contradicting the strict inequality in~\eqref{strict2} (or the fact that $\varphi^\beta_{\Omega,v}>0=\varphi^D_{\Omega,v}$ on $\partial\Omega$). Finally,
$$\lambda_1^{\beta}(\Omega,v)<\lambda_1^D(\Omega,v)$$
for all $\beta>0$.

Let us now turn to the proof of~\eqref{eq:limlambda1}. Pick up any increasing sequence $(\beta_k)_{k\in \N}$ of positive real numbers with $\lim_{k\rightarrow+\infty} \beta_k=+\infty$ and set $\lambda_k:=\lambda_1^{\beta_k}(\Omega,v)$ for all $k\in\N$. The sequence $(\lambda_k)_{k\in \N}$ is increasing and bounded above by $\lambda_1^D(\Omega,v)$ and therefore converges to some $\mu\le \lambda_1^D(\Omega,v)$. For all $k$, if $\varphi_k$ is defined as $\varphi_k:=\theta_k\varphi^{\beta_k}_{\Omega,v}$ with $\theta_k>0$ such that $\|\varphi_k\|_{L^2(\Omega)}=1$, then
\be\label{varphik1}
\left\{
\begin{array}{ll}
-\Delta \varphi_k+v\cdot\nabla\varphi_k=\lambda_k\varphi_k& \mbox{ a.e. in }\Omega,\\
\displaystyle \frac{\partial\varphi_k}{\partial\nu}+\beta_k\varphi_k=0 & \mbox{ on }\partial\Omega.
\end{array}
\right.
\ee
We claim that the sequence $(\varphi_k)_{k\in \N}$ is bounded in $H^1(\Omega)$. Indeed, for all $k\in \N$,
\begin{eqnarray*}
\lambda_k\int_\Omega \varphi_k^2 & = & -\int_\Omega \varphi_k\Delta\varphi_k+\int_\Omega (v\cdot\nabla\varphi_k)\varphi_k\\
& = & \int_\Omega \left\vert \nabla\varphi_k\right\vert^2-\int_{\partial\Omega} \varphi_k\frac{\partial \varphi_k}{\partial\nu}+\int_\Omega (v\cdot\nabla\varphi_k)\varphi_k\\
& = & \int_\Omega \left\vert \nabla\varphi_k\right\vert^2+\beta_k\int_{\partial\Omega} \varphi_k^2+\int_\Omega (v\cdot\nabla\varphi_k)\varphi_k.
\end{eqnarray*}
From this, we derive, for all $\varepsilon>0$,
\be\label{eq:normphik}\baa{rcl}
\displaystyle\int_\Omega \left\vert \nabla\varphi_k\right\vert^2+\beta_k\int_{\partial\Omega}\varphi_k^2 & \le & \displaystyle\lambda_k\int_\Omega \varphi_k^2+\left\Vert v\right\Vert_\infty\int_\Omega \varphi_k\left\vert \nabla\varphi_k\right\vert\\
& \le & \displaystyle\left(\lambda_k+\frac 1{2\varepsilon}\left\Vert v\right\Vert_\infty\right) \int_\Omega \varphi_k^2+\frac{\varepsilon}2\left\Vert v\right\Vert_\infty \int_\Omega \left\vert \nabla\varphi_k\right\vert^2.\eaa
\ee
Provided $\varepsilon \left\Vert v\right\Vert_\infty<2$, recalling that the sequences $(\lambda_k)_{k\in \N}$ and $(\|\varphi_k\|_{L^2(\Omega)})_{k\in\N}$ are bounded, one obtains that the sequence $(\varphi_k)_{k\in \N}$ is bounded in $H^1(\Omega)$, that is, there is $M\in\R_+$ such that $\|\varphi_k\|_{H^1(\Omega)}=\sqrt{\|\varphi_k\|_{L^2(\Omega)}^2+\|\,|\nabla\varphi_k|\,\|_{L^2(\Omega)}^2}\le M$ for all $k\in\N$. Therefore, there exists $\varphi\in H^1(\Omega)$ such that, up to a subsequence, 
\be\label{varphik2}
\varphi_k\rightharpoonup\varphi\mbox{ weakly in } H^1(\Omega),\ \ \varphi_k\rightarrow \varphi\mbox{ strongly in }L^2(\Omega)\ \hbox{ and }\ \varphi_k\rightarrow \varphi\mbox{ a.e. in }\Omega,
\ee
as $k\to+\infty$, whence
\be\label{varphi3}
\|\varphi\|_{L^2(\Omega)}=1\ \hbox{ and }\ \varphi\ge0\hbox{ a.e. in $\Omega$}.
\ee
Moreover, since $\lim_{k\rightarrow+\infty} \beta_k=+\infty$,~\eqref{eq:normphik} shows that 
\be\label{trace0}
\lim_{k\rightarrow+\infty} \mbox{tr}(\varphi_k)=0\mbox{ strongly in }L^2(\partial\Omega),
\ee
where $\mbox{tr }:H^1(\Omega)\rightarrow L^2(\partial\Omega)$ denotes the trace operator. Since this trace operator is continuous from $H^1(\Omega)$ to $L^2(\partial\Omega)$ with the topologies induced by the norms, and since $\varphi_k\rightharpoonup\varphi$ weakly in $H^1(\Omega)$ as $k\to+\infty$, it follows that $\mbox{tr}(\varphi_k)\rightharpoonup\mbox{tr}(\varphi)$ weakly in $L^2(\partial\Omega)$ as $k\to+\infty$. Finally, since $\mbox{tr}(\varphi_k)\to0$ strongly in $L^2(\partial\Omega)$ as $k\to+\infty$ by~\eqref{trace0}, one gets that~$\mbox{tr}(\varphi)=0$, meaning that $\varphi\in H^1_0(\Omega)$.

Consider now $\psi\in C^{\infty}_c(\Omega)$. One has
\begin{eqnarray*}
\int_\Omega \nabla\varphi\cdot\nabla\psi+\int_\Omega (v\cdot\nabla\varphi)\psi & = & \lim_{k\rightarrow+\infty} \int_\Omega \nabla\varphi_k\cdot\nabla\psi+\int_\Omega (v\cdot\nabla\varphi_k)\psi\\
& = & \lim_{k\rightarrow+\infty} \lambda_k\int_\Omega \varphi_k\psi\\
& = &  \mu\int_\Omega \varphi\psi,
\end{eqnarray*}
which means that $\varphi$ is an $H^1_0(\Omega)$ weak solution of
$$
\left\{
\begin{array}{ll}
-\Delta \varphi+v\cdot \nabla \varphi=\mu\varphi & \mbox{ in }\Omega,\vspace{3pt}\\
\mbox{tr}(\varphi)=0 & \mbox{ on }\partial\Omega.
\end{array}
\right.
$$
Elliptic $H^2$ and $W^{2,p}$ estimates show that $\varphi\in W^{2,p}(\Omega)$ for all $1\le p<\infty$, and, since $\varphi\ge0$ in $\Omega$ and $\|\varphi\|_{L^2(\Omega)}=1$, the strong maximum principle entails that $\varphi>0$ in~$\Omega$. Thus, by uniqueness of the principal eigenvalue of $-\Delta+v\cdot\nabla$ under Dirichlet boundary condition, one gets that
$$\mu=\lambda_1^D(\Omega,v),$$
which ends the proof.

Lastly, let us investigate the limit of $\lambda^\beta_1(\Omega,v)$ as $\beta\to0$. Pick up any decreasing sequence $(\beta_k)_{k\in \N}$ of positive real numbers with $\lim_{k\rightarrow+\infty} \beta_k=0$ and set $\lambda_k:=\lambda_1^{\beta_k}(\Omega,v)$ for all $k\in\N$. The sequence $(\lambda_k)_{k\in \N}$ is decreasing and bounded below by $0$, and therefore converges to some $\lambda\ge0$. For all $k$, if $\varphi_k$ is defined as $\varphi_k:=\theta_k\varphi^{\beta_k}_{\Omega,v}$ with $\theta_k>0$ such that $\|\varphi_k\|_{L^2(\Omega)}=1$, then as above~\eqref{varphik1}-\eqref{eq:normphik} still hold and there exists $\varphi\in H^1(\Omega)$ satisfying~\eqref{varphik2}-\eqref{varphi3}, up to a subsequence. Pick now any $\psi\in H^1(\Omega)$. For all $k\in\N$, one has
$$\lambda_k\int_\Omega\varphi_k\psi=\int_\Omega\nabla\varphi_k\cdot\nabla\psi+\beta_k\int_{\partial\Omega}\varphi_k\psi+\int_\Omega(v\cdot\nabla\varphi_k)\psi.$$
But $\beta_k\to0$ as $k\to+\infty$ and the sequence $(\mbox{tr}(\varphi_k))_{k\in\N}$ is bounded in $L^2(\partial\Omega)$ (since so is $(\varphi_k)_{k\in\N}$ in $H^1(\Omega)$). Hence, by~\eqref{varphik2}, the passage to the limit as $k\to+\infty$ in the above formula leads to
$$\lambda\int_\Omega\varphi\psi=\int_\Omega\nabla\varphi\cdot\nabla\psi+\int_\Omega(v\cdot\nabla\varphi)\psi.$$
In other words, $\varphi$ is an $H^1(\Omega)$ weak solution of
$$
\left\{
\begin{array}{ll}
-\Delta \varphi+v\cdot \nabla \varphi=\lambda\varphi & \mbox{ in }\Omega,\\
\displaystyle\frac{\partial\varphi}{\partial\nu}=0 & \mbox{ on }\partial\Omega.
\end{array}
\right.
$$
Elliptic $H^2$ and $W^{2,p}$ estimates show that $\varphi\in W^{2,p}(\Omega)$ for all $1\le p<\infty$, and by~\eqref{varphi3} the strong maximum principle and Hopf lemma entail that $\varphi>0$ in $\overline{\Omega}$. Thus, by uniqueness of the principal eigenvalue of $-\Delta+v\cdot\nabla$ under Neumann boundary condition, one gets that $\lambda=0$. The proof of Lemma~\ref{lem:robindir} is thereby complete.
\end{proof}


\SE{Optimization of the principal eigenvalue in a fixed domain} \label{sec:fixed}

This section is devoted to the proof of Theorem \ref{th:fixedomega}.

\begin{proof} [Proof of Theorem $\ref{th:fixedomega}$]
{\it Part 1.} We first focus on the infimum problem and begin with the existence part. Let $(v_k)_{k\in \N}$ be a sequence of vector fields in $L^{\infty}(\Omega,\R^d)$ such that $\left\Vert v_k\right\Vert_\infty\le \tau$ for all $k$ and 
$$
\lim_{k\rightarrow +\infty} \lambda_1^\beta(\Omega,v_k)=\underline{\lambda}^\beta(\Omega,\tau).
$$
For all $k\in \N$, define $\lambda_k:=\lambda_1^\beta(\Omega,v_k)$ and let $\varphi_k:=\varphi^\beta_{\Omega,v_k}$ be the corresponding eigenfunction, normalized with $\max_{\overline{\Omega}}\varphi_k=1$. Since $(v_k)_{k\in\N}$ is bounded in $L^\infty(\Omega,\R^d)$ and the sequence $(\lambda_k)_{k\in \N}$ is bounded, $W^{2,p}$ elliptic estimates (\cite[Theorem A.29]{YD}) show that the sequence $(\varphi_k)_{k\in \N}$ is bounded in $W^{2,p}(\Omega)$ for all $1\le p<\infty$. Up to a subsequence, there exist $\varphi\in \bigcap_{1\le p<\infty} W^{2,p}(\Omega)$ and $f\in L^\infty(\Omega)$ such that, as $k\to+\infty$,
$$
\varphi_k\rightharpoonup\varphi\mbox{ weakly in } W^{2,p}(\Omega)
$$
for all $1\le p<\infty$,
$$
\varphi_k\rightarrow \varphi\mbox{ strongly in } C^{1,\alpha}(\overline{\Omega})
$$
for all $\alpha\in(0,1)$, and
$$
v_k\cdot\nabla\varphi_k\overset{\ast}{\rightharpoonup}f\mbox{ weak-}\!\ast\mbox{ in }L^\infty(\Omega).
$$
As a consequence,
$$
-\Delta\varphi+f=\underline{\lambda}^\beta(\Omega,\tau)\varphi\ \hbox{ a.e. in $\Omega$}
$$
and
$$
-\Delta\varphi-\tau\left\vert \nabla\varphi\right\vert\le  \underline{\lambda}^\beta(\Omega,\tau)\varphi\ \hbox{ a.e. in $\Omega$}.
$$
Moreover, $\varphi\ge 0$ in $\Omega$, $\left\Vert \varphi\right\Vert_\infty=1$ and
$$\frac{\partial \varphi}{\partial\nu}+\beta\varphi=0\hbox{ on $\partial\Omega$}.$$

Define now $v\in L^{\infty}(\Omega,\R^d)$ by
\be\label{defv}
v(x):=\left\{
\begin{array}{ll}
\displaystyle-\tau\frac{\nabla\varphi(x)}{\left\vert \nabla\varphi(x)\right\vert} & \mbox{ if }\nabla \varphi(x)\ne 0,\\
0 & \mbox{ if }\nabla\varphi(x)=0.
\end{array}
\right.
\ee
Notice that $\left\Vert v\right\Vert_\infty\le\tau$, which entails that $\underline{\lambda}^\beta(\Omega,\tau)\le \lambda_1^\beta(\Omega,v)$. On the one hand,
\be\label{ineqvarphi}
-\Delta\varphi+v\cdot \nabla\varphi=-\Delta\varphi-\tau\left\vert \nabla\varphi\right\vert\le \underline{\lambda}^\beta(\Omega,\tau)\varphi\le\lambda_1^\beta(\Omega,v)\varphi\ \hbox{ a.e. in $\Omega$}.
\ee
On the other hand, $\varphi^\beta_{\Omega,v}>0$ in $\overline{\Omega}$,
$$\displaystyle\frac{\partial \varphi^\beta_{\Omega,v}}{\partial\nu}+\beta\varphi^\beta_{\Omega,v}=0=\displaystyle\frac{\partial \varphi}{\partial\nu}+\beta\varphi\ \hbox{ on $\partial\Omega$}$$
and $\|\varphi^\beta_{\Omega,v}\|_\infty=1$. Lemma~\ref{lem:compar} applied with $(\mu,\beta,\psi,\varphi):=(\lambda_1^\beta(\Omega,v),\beta,\varphi^\beta_{\Omega,\tau},\varphi)$ yields
$$\varphi^\beta_{\Omega,v}=\varphi\ \hbox{ in $\overline{\Omega}$}.$$
As a consequence, all inequalities in~\eqref{ineqvarphi} are equalities and
$$\underline{\lambda}^\beta(\Omega,\tau)=\lambda_1^\beta(\Omega,v).$$
Furthermore, since $\varphi\in W^{2,p}(\Omega)$ for each $1\le p<\infty$, it follows that
$$|\nabla(\partial_{x_i}\varphi)|\times\mathds{1}_{\{\partial_{x_i}\varphi=0\}}=0\ \hbox{ a.e. in }\Omega$$
for each $1\le i\le d$, where $\partial_{x_i}\varphi:=\frac{\partial\varphi}{\partial x_i}$, whence
$$\Delta\varphi\times\mathds{1}_{\{\nabla\varphi=0\}}=0\ \hbox{ a.e. in }\Omega.$$
Since $-\Delta\varphi+v\cdot\nabla\varphi=\lambda^\beta_1(\Omega,v)\varphi>0$ a.e. in $\Omega$, one gets that the set $\{x\in\Omega:\nabla\varphi(x)=0\}$ is negligible. Therefore, in addition to $v\cdot\nabla\varphi=-\tau|\nabla\varphi|$ a.e. in $\Omega$,~\eqref{defv} also entails that $|v(x)|=\tau$ for almost every $x\in\Omega$. The vector field $\underline{v}:=v$ and the function
$$\underline{\varphi}:=\varphi=\varphi^\beta_{\Omega,v}$$
then fulfill the required conclusions of part~1 of Theorem~\ref{th:fixedomega}.

Let us now turn to the uniqueness result in part~1 of Theorem \ref{th:fixedomega}. Assume that $w\in L^\infty(\Omega,\R^d)$ is such that $\|w\|_\infty\le\tau$ and $\lambda_1^\beta(\Omega,w)=\underline{\lambda}^\beta(\Omega,\tau)$. One has
\be\label{ineqvw}\left\{\baa{l}
-\Delta\varphi^\beta_{\Omega,v}+w\cdot\nabla\varphi^\beta_{\Omega,v}\ge-\Delta\varphi^\beta_{\Omega,v}-\tau|\nabla\varphi^\beta_{\Omega,v}|=\underline{\lambda}^\beta(\Omega,\tau)\varphi^\beta_{\Omega,v},\vspace{3pt}\\
-\Delta\varphi^\beta_{\Omega,w}+w\cdot\nabla\varphi^\beta_{\Omega,w}=\underline{\lambda}^\beta(\Omega,\tau)\varphi^\beta_{\Omega,w},\eaa\right.
\ee
a.e. in $\Omega$, together with
$$\displaystyle\frac{\partial\varphi^\beta_{\Omega,v}}{\partial\nu}+\beta\varphi^\beta_{\Omega,v}=0=\displaystyle\frac{\partial\varphi^\beta_{\Omega,w}}{\partial\nu}+\beta\varphi^\beta_{\Omega,w}\ \hbox{ on $\partial\Omega$}.$$
Furthermore, both functions $\varphi^\beta_{\Omega,v}$ and $\varphi^\beta_{\Omega,w}$ are positive (in $\overline{\Omega}$), with $L^\infty$ norms equal to $1$. Lemma~\ref{lem:compar} applied with $(\mu,\beta,\psi,\varphi):=(\underline{\lambda}^\beta(\Omega,\tau),\beta,\varphi^\beta_{\Omega,v},\varphi^\beta_{\Omega,w})$, and the vector field $w$ instead of $v$, then entails
$$\varphi^\beta_{\Omega,v}=\varphi^\beta_{\Omega,w}\ \hbox{ in $\overline{\Omega}$}.$$
Consequently, the first line in~\eqref{ineqvw} then yields
$$w\cdot\nabla\varphi^\beta_{\Omega,v}=-\tau|\nabla\varphi^\beta_{\Omega,v}|\ \hbox{ a.e. in $\Omega$},$$
that is, $w\cdot\nabla\varphi=-\tau|\nabla\varphi|$ a.e. in $\Omega$. Since $\nabla\varphi\neq0$ a.e. in $\Omega$ and $\|w\|_\infty\le\tau$, one concludes that
$$w=-\tau\frac{\nabla\varphi}{|\nabla\varphi|}\ \hbox{ a.e. in $\Omega$},$$
that is, $w=v$ a.e. in $\Omega$.

Lastly, let $\lambda\in\R$ and $\phi\in\bigcap_{1\le p<\infty}W^{2,p}(\Omega)$ satisfy
$$
\left\{
\begin{array}{ll}
-\Delta \phi-\tau\,|\nabla\phi|=\lambda\phi\ \hbox{ and }\ \phi\ge0 & \mbox{in }\Omega,\\
\displaystyle\frac{\partial \phi}{\partial\nu}+\beta \phi=0 & \mbox{on }\partial\Omega,
\end{array}
\right.
$$
and $\max_{\overline{\Omega}}\phi=1$. Define $q\in L^{\infty}(\Omega,\R^d)$ by
$$q(x):=\left\{
\begin{array}{ll}
\displaystyle-\tau\frac{\nabla\phi(x)}{\left\vert \nabla\phi(x)\right\vert} & \mbox{ if }\nabla \phi(x)\ne 0,\\
0 & \mbox{ if }\nabla\phi(x)=0.
\end{array}
\right.$$
Notice that $\left\Vert q\right\Vert_\infty\le\tau$. Since $-\tau|\nabla\phi|=q\cdot\nabla\phi$ a.e. in $\Omega$, the nonnegativity of~$\phi$ and the uniqueness of the pair of principal eigenvalue and principal normalized eigenfunction imply that
$$\lambda=\lambda^\beta_1(\Omega,q)\ge\underline{\lambda}^\beta(\Omega,\tau),\ \hbox{ and }\ \phi=\varphi^\beta_{\Omega,q}\hbox{ in }\overline{\Omega}.$$
Both functions $\phi=\varphi^\beta_{\Omega,q}$ and $\varphi=\varphi^\beta_{\Omega,v}$ are positive in $\overline{\Omega}$ with $L^\infty$ norms equal to $1$, and they satisfy
\be\label{ineq3}\left\{\baa{l}
-\Delta\varphi^\beta_{\Omega,q}+v\cdot\nabla\varphi^\beta_{\Omega,q}\ge-\Delta\varphi^\beta_{\Omega,q}-\tau|\nabla\varphi^\beta_{\Omega,q}|=\lambda^\beta_1(\Omega,q)\varphi^\beta_{\Omega,q},\vspace{3pt}\\
-\Delta\varphi^\beta_{\Omega,v}+v\cdot\nabla\varphi^\beta_{\Omega,v}=\underline{\lambda}^\beta(\Omega,\tau)\varphi^\beta_{\Omega,v}\le\lambda^\beta_1(\Omega,q)\varphi^\beta_{\Omega,v},\eaa\right.
\ee
a.e. in $\Omega$, together with
$$\displaystyle\frac{\partial\varphi^\beta_{\Omega,q}}{\partial\nu}+\beta\varphi^\beta_{\Omega,q}=0=\displaystyle\frac{\partial\varphi^\beta_{\Omega,v}}{\partial\nu}+\beta\varphi^\beta_{\Omega,v}\ \hbox{ on $\partial\Omega$}.$$
Lemma~\ref{lem:compar} applied with $(\mu,\beta,\psi,\varphi)=(\lambda^\beta_1(\Omega,q),\beta,\varphi^\beta_{\Omega,q},\varphi^\beta_{\Omega,v})$ then entails
$$\varphi^\beta_{\Omega,q}=\varphi^\beta_{\Omega,v}\ \hbox{ in $\overline{\Omega}$},$$
that is, $\phi=\varphi=\underline{\varphi}$ in $\overline{\Omega}$. Furthermore, all inequalities in~\eqref{ineq3} are equalities and
$$\underline{\lambda}^\beta(\Omega,\tau)=\lambda^\beta_1(\Omega,q),$$
whence $\lambda=\underline{\lambda}^\beta(\Omega,\tau)$. All properties in part~1 of Theorem~\ref{th:fixedomega} have now been proved.

\vskip 0.3cm
\noindent{\it Part 2.} Notice that, for all $v\in L^{\infty}(\Omega)$, $\lambda_1^\beta(\Omega,v)\le \lambda_1^D(\Omega,v)$ by Lemma~\ref{lem:robindir}. Since
$$\sup_{v\in L^\infty(\Omega,\R^d),\,\|v\|_\infty\le\tau}\lambda_1^D(\Omega,v)<+\infty$$
by~\cite[Proposition~5.1]{BNV},~\cite[Theorem~1.5]{HNRpreprint} or~\cite[Theorem~6.6]{HNR}, if follows that the quantity $\overline{\lambda}^\beta(\Omega,\tau)$ defined in~\eqref{overlambda} is a real number. Then, arguments similar to those in part~1 above yield the conclusions of part~2.

\vskip 0.3cm
\noindent{\it Part 3.} Consider now the case $\Omega=\Omega^\ast$ and denote
$$\phi:=\varphi^\beta_{\Omega^\ast,\tau e_r}.$$
This function $\phi$ is positive in $\overline{\Omega^\ast}$, it is of class $W^{2,p}(\Omega^\ast)$ for all $1\le p<\infty$, and $\max_{\overline{\Omega^\ast}}\phi=1$. For any $\mathcal{R}\in O(d)$, the function $\phi\circ\mathcal{R}$ satisfies the same equation as~$\phi$ in $\Omega^\ast$ and the same boundary condition on $\partial\Omega^\ast$. The uniqueness of the pair of eigenvalue and principal normalized eigenfunction then entails that $\phi\circ\mathcal{R}=\phi$ in $\overline{\Omega^\ast}$ for any $\mathcal{R}\in O(d)$, that is, $\phi$ is radially symmetric in $\overline{\Omega^*}$. Let $R$ denote the radius of~$\Omega^\ast$. For any $\sigma\in(0,R]$, there holds
$$-\Delta\phi+\tau e_r\cdot\nabla\phi=\lambda_1^\beta(\Omega^\ast,\tau e_r)\phi>0$$
almost everywhere in $\{x:|x|\le\sigma\}$ and $\phi$ is constant on the sphere $\{y:|y| =\sigma\}$. The weak maximum principle then implies that $\phi(x)\ge\phi(y)$ for all $|x|\le|y|=\sigma$, and the Hopf lemma even yields $e_r\cdot\nabla\phi(y)<0$ for all $|y|=\sigma$. As a conclusion, $\phi$ is radially decreasing and
$$\tau e_r\cdot\nabla\phi=-\tau|\nabla\phi|$$
everywhere in $\overline{\Omega^\ast}\setminus\{0\}$ and the function $\phi$ then fulfills~\eqref{eqphi} in $\Omega^\ast$ with $\lambda:=\lambda^\beta_1(\Omega^\ast,\tau e_r)$. It then follows from the last result of part~1 of the present theorem that
$$\underline{\lambda}^\beta(\Omega^\ast,\tau)=\lambda^\beta_1(\Omega^*,\tau e_r),$$
and the uniqueness of the vector field minimizing $\lambda^\beta_1(\Omega^\ast,v)$ implies that $\underline{v}=\tau e_r$.

Similarly, by denoting $\psi:=\varphi^\beta_{\Omega^\ast,-\tau e_r}$, one proves similarly that $\psi$ is radially decreasing and one still has $\tau e_r\cdot\nabla\psi=-\tau|\nabla\psi|$, that is, $-\tau e_r\cdot\nabla\psi=\tau|\nabla\psi|$, everywhere in $\overline{\Omega^\ast}\setminus\{0\}$. Part~2 of the present theorem then implies that
$$\overline{\lambda}^\beta(\Omega^\ast,\tau)=\lambda^\beta_1(\Omega^*,-\tau e_r)$$
and $\overline{v}=-\tau e_r$. The proof of Theorem~\ref{th:fixedomega} is thereby complete.
\end{proof}


\SE{Proof of the minimization result} \label{sec:main}

Let us now prove Theorem \ref{th:main}. Arguing by contradiction, assume that the conclusion does not hold. There exist then a sequence $(\beta_k)_{k\in \N}$ of positive numbers such that $\lim_{k\rightarrow+\infty} \beta_k=+\infty$ and a sequence of vector fields $(v_k)_{k\in \N}$ such that, for all $k\in\N$, $\left\Vert v_k\right\Vert_{\infty}\le \tau$ and
\begin{equation} \label{eq:ineqlambdak}
\lambda_1^{\beta_k}(\Omega,v_k)<\lambda_1^{\beta_k}(\Omega^{\ast},\tau e_r)+\frac 1{k+1}.
\end{equation}
For all $k\in \N$, write
$$\varphi_k:=\varphi^{\beta_k}_{\Omega,v_k}\ \hbox{ and }\ \lambda_k:=\lambda_1^{\beta_k}(\Omega,v_k).$$
One has
$$
\left\{
\begin{array}{ll}
-\Delta\varphi_k+v_k\cdot\nabla\varphi_k=\lambda_k\varphi_k & \mbox{ a.e. in }\Omega,\\
\displaystyle\frac 1{\beta_k}\frac{\partial\varphi_k}{\partial\nu}=-\varphi_k& \mbox{ on }\partial\Omega.
\end{array}
\right.
$$
Lemma \ref{lem:robindir} shows that $\lambda_k\le \lambda_1^D(\Omega,v_k)$ for all $k\in\N$, while \cite[Proposition~5.1]{BNV} ensures that the sequence $(\lambda_1^D(\Omega,v_k))_{k\in \N}$ is bounded (recall that $\left\Vert v_k\right\Vert_{\infty}\le \tau$ for all~$k\in\N$). Furthermore, each $\lambda_k$ is a positive real number. Therefore, the sequence~$(\lambda_k)_{k\in\N}$ is bounded. Arguing as in the proof of Lemma \ref{lem:robindir}, one concludes that the sequence~$(\varphi_k)_{k\in \N}$ is then bounded in~$H^1(\Omega)$, which entails that the sequence $(\mbox{tr}(\varphi_k))_{k\in\N}$ is bounded in $H^{\frac 12}(\partial\Omega)$. Therefore, together with the boundedness of the sequence $(1/\beta_k)_{k\in\N}$, \cite[Theorem~15.2]{ADN} implies that the sequence $(\varphi_k)_{k\in \N}$ is bounded in $W^{2,2}(\Omega)$, and a bootstrap argument therefore shows that $(\varphi_k)_{k\in \N}$ is bounded in~$W^{2,p}(\Omega)$ for all $1\le p<\infty$. Thus, there exist $\mu\in\R$, $\varphi\in\bigcap_{1\le p<\infty}W^{2,p}(\Omega)$ and $f\in L^{\infty}(\Omega)$ such that, up to a subsequence,
$$
\lim_{k\rightarrow +\infty} \lambda_k=\mu,
$$
$$
\varphi_k\mathop{\rightharpoonup}_{k\to+\infty}\varphi\mbox{ weakly in }W^{2,p}(\Omega)\mbox{ and } \varphi_k\mathop{\to}_{k\to+\infty}\varphi\mbox{ strongly in }C^{1,\alpha}(\overline{\Omega})
$$
for all $1\le p<\infty$ and all $\alpha\in (0,1)$, and
$$
v_k\cdot\nabla\varphi_k\rightharpoonup f\mbox{ weakly-}\!\ast\mbox{ in }L^\infty(\Omega).
$$
Furthermore, as in the proof of Lemma~\ref{lem:robindir}, there holds $\mbox{tr}(\varphi_k)\to0$ as $k\to+\infty$ strongly in $L^2(\partial\Omega)$. One therefore has
$$
\left\{
\begin{array}{ll}
-\Delta\varphi+f=\mu\varphi & \mbox{a.e. in }\Omega,\\
\varphi\ge 0 & \mbox {in }\overline{\Omega},\\
\mbox{tr}(\varphi)=0 & \mbox{on }\partial\Omega,\\
\displaystyle\max_{\overline{\Omega}} \varphi=1
\end{array}
\right.
$$
and $f\ge -\tau\left\vert \nabla\varphi\right\vert$ a.e. in $\Omega$, so that
$$
-\Delta\varphi-\tau\left\vert \nabla\varphi\right\vert\le \mu\varphi\ \mbox{ a.e. in }\Omega.
$$
Define
$$
v(x):=\left\{
\begin{array}{ll}
\displaystyle -\tau\frac{\nabla\varphi(x)}{\left\vert \nabla\varphi(x)\right\vert} & \mbox{ if }\nabla\varphi(x)\ne 0,\\
0 & \mbox{ if }\nabla\varphi(x)=0,
\end{array}
\right.
$$
so that $\left\Vert v\right\Vert_\infty\le \tau$ and
$$
-\Delta\varphi+v\cdot\nabla\varphi\le \mu\varphi\ \mbox{ a.e. in }\Omega.
$$

Let now $\psi:=\varphi^D_{\Omega,v}$, so that $\psi>0$ in $\Omega$ and
$$
-\Delta\psi+v\cdot\nabla\psi=\lambda_1^D(\Omega,v)\psi\ \mbox{ a.e. in }\Omega.
$$
If $\mu<\lambda_1^D(\Omega,v)$, then
$$
-\Delta\varphi+v\cdot\nabla\varphi\le \mu\varphi\le \lambda_1^D(\Omega,v)\varphi\ \mbox{ a.e. in }\Omega,
$$
and \cite[Lemma 2.1]{HNRpreprint} implies that $\varphi=\psi$ in $\overline{\Omega}$, therefore $\mu=\lambda_1^D(\Omega,v)$, a contradiction. Finally,
$$\lambda_1^D(\Omega,v)\le \mu.$$
But \eqref{eq:ineqlambdak} and Lemma \ref{lem:robindir} imply that
$$
\mu\le \lambda_1^D(\Omega^{\ast},\tau e_r),
$$
and one therefore obtains
$$
\lambda_1^D(\Omega,v)\le \lambda_1^D(\Omega^{\ast},\tau e_r),
$$
which contradicts the ``equality'' statement in \cite[Theorem 1.1]{HNRpreprint} since $\Omega$ is not a ball. This concludes the proof of Theorem \ref{th:main}.\hfill$\Box$

\bibliographystyle{plain}
{\small{\bibliography{Biblio-HR-Robin}}}

\end{document}